\documentclass[11pt]{article}

\usepackage[T1]{fontenc}
\usepackage[utf8]{inputenc}
\usepackage{lmodern}
\usepackage{microtype}
\usepackage[a4paper,margin=30mm]{geometry}
\usepackage{amsmath,amssymb,amsthm,mathtools,mathrsfs,needspace,array}
\usepackage{aliascnt}
\usepackage{xcolor}
\usepackage{hyperref}
\usepackage[nameinlink,capitalise,noabbrev]{cleveref}
\usepackage{fancyhdr}
\hypersetup{
  colorlinks=true,linkcolor=blue!45!black,citecolor=green!35!black,
  urlcolor=blue!55!black,
  pdftitle={Movable Curves and Ample Normal Bundles},
  pdfsubject={Movable curves, ample normal sheaves, and smooth-curve replacements},
  pdfkeywords={movable curves, ample normal bundle, normal sheaf, birational geometry}
}
\newtheorem{theorem}{Theorem}[section]
\newaliascnt{proposition}{theorem}
\newtheorem{proposition}[proposition]{Proposition}
\aliascntresetthe{proposition}
\newaliascnt{lemma}{theorem}
\newtheorem{lemma}[lemma]{Lemma}
\aliascntresetthe{lemma}
\newaliascnt{corollary}{theorem}

\aliascntresetthe{corollary}
\theoremstyle{definition}
\newaliascnt{definition}{theorem}
\newtheorem{definition}[definition]{Definition}
\aliascntresetthe{definition}
\newaliascnt{question}{theorem}
\newtheorem{question}[question]{Question}
\aliascntresetthe{question}
\crefname{question}{Question}{Questions}
\Crefname{question}{Question}{Questions}
\newaliascnt{example}{theorem}
\newtheorem{example}[example]{Example}
\aliascntresetthe{example}
\theoremstyle{remark}
\newaliascnt{remark}{theorem}
\newtheorem{remark}[remark]{Remark}
\aliascntresetthe{remark}

\newcommand{\C}{\mathbb C}
\newcommand{\Q}{\mathbb Q}
\newcommand{\R}{\mathbb R}
\newcommand{\Z}{\mathbb Z}
\newcommand{\cO}{\mathcal O}
\newcommand{\cN}{\mathcal N}
\newcommand{\cI}{\mathcal I}
\newcommand{\cA}{\mathcal A}
\newcommand{\cG}{\mathcal G}
\newcommand{\cS}{\mathcal S}
\newcommand{\cF}{\mathcal F}
\newcommand{\Mov}{\overline{\operatorname{Mov}}_1}
\newcommand{\Eff}{\overline{\operatorname{Eff}}}
\newcommand{\Nef}{\operatorname{Nef}}
\newcommand{\Supp}{\operatorname{Supp}}
\newcommand{\Hom}{\mathcal Hom}
\newcommand{\tors}{\mathrm{tors}}
\newcommand{\Proj}{\operatorname{Proj}}
\newcommand{\CI}{\operatorname{CI}_{\mathrm{nef}}}
\newcommand{\Asm}{\cA_{\mathrm{sm}}}
\newcommand{\Bl}{\operatorname{Bl}}
\newcommand{\Pj}{\mathbb P}
\newcommand{\Psm}{\mathcal P_{\mathrm{sm}}}

\title{\bfseries Movable Curves and Ample Normal Bundles}
\author{Jingcao Wu\thanks{School of Mathematics, Shanghai University of Finance and Economics, Shanghai 200433, People's Republic of China. E-mail: \texttt{wujincao@shufe.edu.cn}.}}
\date{}

\begin{document}
\maketitle

\begin{abstract}
We prove that the movable cone of a smooth complex projective variety is the closed cone generated by possibly singular curves with ample coherent normal sheaves, answering a question of Boucksom--Demailly--P\u{a}un--Peternell. For smooth curves, we establish the analogous equality in several geometric settings and give explicit constructions on blowups.
\end{abstract}



\section{Introduction}

Boucksom--Demailly--P\u{a}un--Peternell (BDPP) identified the closed movable cone of curves with the dual of the pseudoeffective divisor cone \cite{BDPP13}. Their Theorem~2.4 also identifies it with the closed cone generated by curves with nef normal sheaves. In the discussion following Corollary~2.5, they asked whether the same equality holds for curves with ample normal bundles.

Throughout, varieties are defined over $\C$. Let $X$ be a smooth projective variety of dimension $n\ge2$, and let $\Mov(X)$ denote its closed movable curve cone. We use the normal sheaf convention of BDPP and Peternell \cite{BDPP13,Pet12Ample}, allowing reduced irreducible curves to be singular.

\begin{definition}
For a reduced irreducible curve $C\subset X$, put
\[
 \cN_{C/X}:=\Hom_{\cO_C}(\cI_C/\cI_C^2,\cO_C).
\]
A coherent sheaf $\cF$ on a projective curve is called ample if $\cO_{\mathbb P_C(\cF)}(1)$ is ample, where $\mathbb P$ parametrizes one-dimensional quotients. Define
\[
\cA_{\mathrm{sh}}(X):=\overline{\operatorname{cone}}\{[C]: C\subset X\text{ reduced irreducible},\ \cN_{C/X}\text{ ample}\}.
\]
\end{definition}
For smooth $C$, this sheaf is the usual normal bundle $N_{C/X}$. Our first theorem establishes BDPP's expected equality.

\begin{theorem}\label{thm:sheaf-cone}
For every smooth complex projective variety $X$,
\begin{equation}
\cA_{\mathrm{sh}}(X)=\Mov(X).
\end{equation}
\end{theorem}

Let $\Asm(X)$ be the closed convex cone generated by smooth irreducible embedded curves $C\subset X$ with ample normal bundle $N_{C/X}$. Requiring the curves to be smooth leads to the following stronger question.

\begin{question}\label{q:main}
For every smooth complex projective variety $X$, does
\[
 \Mov(X)=\Asm(X)
\]
hold?
\end{question}

To distinguish cone generation from approximation by individual curves, define
\[
\Psm(X)=\overline{\{a[C]:a>0,\ C\subset X\text{ smooth irreducible},\ N_{C/X}\text{ ample}\}}.
\]
We do not assume that $\Psm(X)$ is convex; its closed convex hull is $\Asm(X)$. Equality $\Psm(X)=\Mov(X)$ therefore gives the stronger conclusion that each movable class is a limit of rescaled individual curve classes.

\noindent\textbf{Related work.}
Peternell studied subvarieties with ample normal bundles and conditions under which their curve classes lie in the interior of the Mori cone \cite{Pet12Ample}. Ottem proved that a smooth curve with ample normal bundle has class in the interior of the Mori cone; if the curve is also strictly nef, its class lies in the interior of the movable cone \cite[Theorem~2]{Ottem16}. Lau extended these methods to locally ample curves, including singular settings, and obtained interior movability under additional intersection hypotheses \cite{Lau18}. Fulger developed Seshadri constants for curve classes and gave a proof of the interior-Mori statement valid in arbitrary characteristic \cite{Fulger21}.

Fulger and Lehmann introduced positive cones of numerical cycle classes \cite{FL17}. Their construction in \cite[Example~5.11]{FL17} yields the Mori dream space case considered here: the complete-intersection curves on small modifications lie in the common isomorphism locus and have ample normal bundles. We record this consequence together with an application of the minimal-model theory of Birkar--Cascini--Hacon--McKernan \cite{BCHM10}. Witt Nystr\"om proved the projective transcendental duality between pseudoeffective and movable classes and identified the latter with the closure of balanced metric classes \cite{WittNystrom19}. \Cref{q:main} asks whether the movable curve cone admits generators given by smooth embedded curves with ample normal bundles.

\noindent\textbf{The approach.}
By BDPP duality, the movable cone is the closed convex cone generated by strongly movable classes
\begin{equation}\label{eq:strongly-movable}
\mu_*[A_1\cdots A_{n-1}],\qquad\mu:Y\longrightarrow X\text{ birational},\quad A_i\text{ very ample on a smooth }Y.
\end{equation}
We first show that the image of a smooth complete-intersection curve has ample coherent normal sheaf, even when the image is singular. This proves \cref{thm:sheaf-cone}. For the smooth-curve question, we must also control the singularities of the image while preserving its numerical class and the ampleness of its normal bundle. We do this in several settings using incidence geometry, moving families, and explicit constructions in exceptional divisors. The remaining inclusion in general is
\begin{equation}
 \Mov(X)\stackrel{?}{\subseteq}\Asm(X).
\end{equation}

The following theorem collects the cases in which we obtain a positive answer to \cref{q:main}.

\begin{theorem}
The equality
\begin{equation}
 \cA_{\mathrm{sm}}(X)=\Mov(X)
\end{equation}
holds when $X$ is a smooth Mori dream space, or when there is an effective $\Q$-divisor $\Delta_0$ such that $(X,\Delta_0)$ is klt and $-(K_X+\Delta_0)$ is nef. For $\dim X\ge3$, it also holds when $X$ carries an action of a positive-dimensional abelian variety with no zero-dimensional orbit, or when $X=Z\times\mathbb P^r$ with $\dim Z,r\ge1$. 
\end{theorem}

\noindent\emph{Explicit blowup cones.}
For the blowup results, we also compute the cones and construct individual curve representatives. Suppose that $\rho(X)=1$, that $L$ is ample, and that $Z\subset X$ is a smooth connected positive-dimensional scheme-theoretic complete intersection of ample divisors $D_i\equiv a_iL$ with $0<a_1\le\cdots\le a_r$ and $2\le r\le\dim X-1$. For $Y=\Bl_Z X$, write $H=\pi^*L$, let $E$ be the exceptional divisor, and put $h=H\cdot\alpha$ and $e=E\cdot\alpha$. Then \cref{thm:ci-all-codim} proves the explicit identity
\[
\Asm(Y)=\Mov(Y)=\{(h,e):h\ge0,\ 0\le e\le a_1h\}.
\]
Every rational interior ray has an individual smooth ample-normal representative. For smooth transverse codimension-two complete intersections, \cref{prop:ci-two-cones} further identifies this cone with the nef complete-intersection cone.

The proofs combine comparison of normal sheaves under birational maps, incidence arguments on normal sources, and curves in exceptional ruled surfaces. They require neither the vanishing $H^1(C,N_{C/X})=0$ nor a general smoothing theorem for singular embedded curves.

\Cref{sec:normal-sheaf} proves \cref{thm:sheaf-cone}. \Cref{sec:smooth-embeddings} develops the embedding criterion and birational modifications used in \cref{sec:positive-cases} to treat minimal models and blowups. \Cref{sec:moving-families} treats moving families, products, and finite descent. Finally, \cref{sec:conditional-replacement} gives conditional methods for smooth replacement using point-saturated models and deformations of centres.

\section{Ampleness on curves and birational normal comparison}\label{sec:normal-sheaf}
We prove the coherent-normal-sheaf formulation of BDPP's expectation by comparing normal sheaves under birational morphisms. The comparison preserves ampleness even when the image curve is singular.

\begin{lemma}\label{lem:torsion}
Let $C$ be a smooth complex projective curve, and let
\[
 0\longrightarrow E\longrightarrow F\longrightarrow T\longrightarrow0
\]
be an exact sequence with $E,F$ vector bundles of the same rank and $T$ torsion. If $E$ is ample, so is $F$.
\end{lemma}

\begin{proof}
Dualizing gives $F^\vee\hookrightarrow E^\vee$. Thus $\mu_{\max}(F^\vee)\leq\mu_{\max}(E^\vee)$, or $\mu_{\min}(F)\geq\mu_{\min}(E)>0$. In characteristic zero the positivity of the smallest Harder--Narasimhan slope is equivalent to ampleness on a smooth projective curve; see \cite{Har66}.
\end{proof}

\begin{lemma}\label{lem:coherent-descent}
Let $\nu:\widetilde C\to C$ be a finite surjective morphism of projective curves, with $\widetilde C$ smooth and connected. Let $\cF$ be a coherent sheaf on $C$, generically locally free. If
\[
 (\nu^*\cF)/\tors
\]
is ample on $\widetilde C$, then $\cF$ is ample.
\end{lemma}

\begin{proof}
Set $G=\nu^*\cF$, $T=\tors(G)$ and $E=G/T$. Since $E$ is locally free and $T$ has zero-dimensional support, $\operatorname{Ext}^1(E,T)=H^1(E^\vee\otimes T)=0$; hence $G\simeq E\oplus T$. The disjoint union of $\mathbb P(E)$ and the finitely many full projective fibers $\mathbb P(G\otimes k(p))$ at $\Supp T$ maps finitely and surjectively to $\mathbb P(G)$. Its pullback of $\cO_{\mathbb P(G)}(1)$ is ample on each component: on
$\mathbb P(E)$ by assumption, and on every fiber by the usual $\cO_{\mathbb P^r}(1)$. Finite descent of ampleness gives that $G$ is ample. The base change $\mathbb P(G)=\widetilde C\times_C\mathbb P(\cF)$ is finite and surjective over $\mathbb P(\cF)$; applying finite descent once more proves that $\cF$ is ample.
\end{proof}

\begin{lemma}\label{lem:birational-sheaf}
Let $\mu:Y\to X$ be a birational morphism of smooth projective varieties. Let $\Gamma\subset Y$ be a smooth irreducible curve meeting the isomorphism locus of $\mu$, and let $C=\mu(\Gamma)_{\mathrm{red}}$. If $N_{\Gamma/Y}$ is ample, then $\cN_{C/X}$ is ample as a coherent sheaf.
\end{lemma}

\begin{proof}
The restriction $\nu=\mu|_\Gamma:\Gamma\to C$ is proper, finite and birational, hence the normalization. Put
\[
 E=(\nu^*\cN_{C/X})/\tors.
\]
The natural map from ambient derivations to embedded normal derivations, composed with $d\mu$, gives
\[
 T_Y|_\Gamma\longrightarrow\nu^*\cN_{C/X}\longrightarrow E.
\]
On the dense open subset where $\mu$ is an isomorphism the map vanishes on $T_\Gamma$. Because $E$ is torsion-free on a smooth curve, it vanishes on $T_\Gamma$ everywhere and factors through
\[
 N_{\Gamma/Y}\longrightarrow E.
\]
This map is an isomorphism generically. Since the two bundles have rank $n-1$, it is injective with torsion cokernel. By \cref{lem:torsion}, $E$ is ample. Apply
\cref{lem:coherent-descent} to $\cN_{C/X}$.
\end{proof}

\begin{proof}[Proof of \cref{thm:sheaf-cone}]
Take a strongly movable generator \eqref{eq:strongly-movable}, and choose general $D_i\in|A_i|$. Bertini gives a smooth irreducible curve
\[
 \Gamma=D_1\cap\cdots\cap D_{n-1}
\]
which meets the isomorphism locus of $\mu$. Its normal bundle is
\[
 N_{\Gamma/Y}\simeq\bigoplus_{i=1}^{n-1}\cO_\Gamma(A_i),
\]
and is ample. The image $C$ has ample coherent normal sheaf by \cref{lem:birational-sheaf}, while $[C]=\mu_*[\Gamma]$. Taking closed convex cones proves $\Mov(X)\subseteq\cA_{\mathrm{sh}}(X)$. Conversely, an ample coherent sheaf is nef, so \cite[Theorem~2.4]{BDPP13} gives $\cA_{\mathrm{sh}}(X)\subseteq\Mov(X)$.
\end{proof}

\section{Smooth embeddings and birational modifications}\label{sec:smooth-embeddings}
A general curve on a birational model may acquire singularities under contraction. We give a criterion for a general complete intersection on a normal source to have a smooth embedded image, then apply it to several birational modifications.

\subsection{Smooth embeddings from normal sources}
The following criterion will also be used for moving families in \cref{sec:moving-families}.

\begin{lemma}\label{lem:larger-source}
Let $g:Z\to X$ be a dominant morphism, where $Z$ is normal integral projective of dimension $N$, and $X$ is smooth projective of dimension $n\ge3$. Suppose no prime divisor of $Z$ maps to a point of $X$. If $A_1,\ldots,A_{N-1}$ are very ample, a general complete-intersection curve
\[
 C=A_1\cap\cdots\cap A_{N-1}
\]
lies in $Z_{\rm reg}$, maps as a closed embedding to $X$, and has ample normal bundle there. Consequently
\[
 g_*(L_1\cdots L_{N-1})\in\Psm(X)
\]
for any nef real divisor classes $L_i$ on $Z$.
\end{lemma}

\begin{proof}
A general complete-intersection curve is irreducible and avoids $\operatorname{Sing}Z$, whose codimension is at least two. Put $\mathcal H=\prod_i|A_i|$.

We first exclude two distinct points having the same image. Stratify $X$ into finitely many locally closed pieces on which the dimensions of the components of the fibers of $g$ are constant. Consider a component of the off-diagonal relation
\[
 R=(Z\times_X Z)\setminus\Delta.
\]
If it dominates $X$, its dimension is $2N-n\le2N-3$. If its image closure $B$ is proper of positive dimension $b$, each relevant component of $g^{-1}(B)$ has dimension at most
$N-1$. Thus the two generic fiber dimensions $r_1,r_2$ satisfy $r_i+b\le N-1$, and
\[
 \dim R_{\!B}\le b+r_1+r_2\le2N-b-2\le2N-3.
\]
Over a zero-dimensional stratum the hypothesis gives fiber dimension at most $N-2$, so the relation has dimension at most $2N-4$. Requiring each defining divisor to contain a
prescribed pair of distinct points imposes $2(N-1)$ independent conditions. All bad-pair incidences therefore have dimension smaller than $\dim\mathcal H$.

For tangents, stratify $Z_{\rm reg}$ into smooth locally closed rank strata for $dg$. At a rank-$k$ point, the conditions that the complete intersection pass through that point and have tangent line in $\ker dg$ have total codimension
\[
 (N-1)+k
\]
among tuples with independent first derivatives. Very ampleness supplies the point and first-jet conditions. The rank-$n$ locus has dimension at most $N$, giving an incidence excess at most $1-n<0$. Rank $1\le k<n$ loci are proper and have dimension at most $N-1$, giving excess at most $-k<0$. Finally, every irreducible rank-zero stratum maps to a point in characteristic zero. The hypothesis on contracted divisors bounds its dimension by $N-2$, giving excess at most $-1$. Thus $dg$ is injective on the tangent line to $C$ at every point.

The resulting proper injective immersion is a closed embedding. The normal bundle
\[
 N_{C/Z}=\bigoplus_i\cO_C(A_i)
\]
is ample. Its image in $N_{g(C)/X}$ is a locally free quotient on the smooth curve $C$, and is generically of full target rank: a general $C$ is not contained in the critical locus of $g$. The quotient by this image is torsion. An ample bundle on a smooth curve stays ample under such a torsion enlargement, so $N_{g(C)/X}$ is ample.

For real nef classes $L_i$, approximate each by rational ample classes and take sufficiently divisible very ample multiples. Positive rescaling and the closedness of $\Psm(X)$ give the result.
\end{proof}

In particular, \cref{lem:larger-source} applies when \(\mu:Y\to X\) is a birational morphism of smooth projective \(n\)-folds, \(n\geq3\), that contracts no prime divisor to a point. In this case, a general complete-intersection curve of very ample divisors \(A_1,\ldots,A_{n-1}\) maps isomorphically onto a smooth curve with ample normal bundle. Hence,
\[
  \mu_*(A_1\cdots A_{n-1})\in\Asm(X).
\]

\subsection{Birational modifications}
The next results remove certain point-supported exceptional contributions without changing the relevant numerical curve classes.

\begin{proposition}
Let $\pi:Y\to Z$ be a birational morphism of smooth projective $n$-folds that is an isomorphism outside a finite subset of $Z$. For ample $\Q$-Cartier divisors $A_1,\ldots,A_{n-1}$ on $Y$, the divisors $H_i=\pi_*A_i$ are ample and
\begin{equation}\label{eq:point-exact}
 \pi_*(A_1\cdots A_{n-1})=H_1\cdots H_{n-1}.
\end{equation}
For nef real divisor classes $A_i$, the classes $L_i=\pi_*A_i$ are nef and the same product identity holds with $L_i$ in place of $H_i$. Moreover, if $n\geq3$ and $\nu:Z\to X$ is a dominant morphism to a smooth projective $n$-fold such that no prime divisor of $Z$ maps to a point, then
\[
 (\nu\pi)_*(A_1\cdots A_{n-1})\in\Psm(X)
\]
for nef real classes $A_i$.
\end{proposition}

\begin{proof}
First suppose the $A_i$ are ample rational classes. Clear denominators and choose effective representatives with no exceptional component. Then $\pi^*H_i=A_i+F_i$ with $F_i\geq0$ exceptional. If $B\subset Z$ is an irreducible curve and $\widetilde B$ its strict transform, then
\[
 H_i\cdot B=A_i\cdot\widetilde B+F_i\cdot\widetilde B>0,
\]
because $\widetilde B\not\subset\operatorname{Supp}F_i$. Thus $H_i$ is nef. For a positive-dimensional subvariety $W\subset Z$ of dimension $k$, expand
\[
(\pi^*H_i)^k-A_i^k=F_i\sum_{j=0}^{k-1}(\pi^*H_i)^{k-1-j}A_i^j.
\]
The strict transform $\widetilde W$ is not contained in $F_i$; the restriction $F_i|_{\widetilde W}$ is effective, and both $A_i$ and $\pi^*H_i$ are nef. Hence $H_i^k\cdot W\geq A_i^k\cdot\widetilde W>0$. Nakai--Moishezon proves that $H_i$ is ample. Expanding
$A_i=\pi^*H_i-F_i$, every mixed product containing an $F_i$ is a one-cycle supported over the finite exceptional image and pushes forward to zero. This proves \eqref{eq:point-exact}.

For nef real $A_i$, approximate each class by ample rational classes. Their pushforwards are ample by the preceding argument, so $L_i$ is nef. The mixed-product identity passes to the limit. Finally,
\[
(\nu\pi)_*(A_1\cdots A_{n-1})=\nu_*(L_1\cdots L_{n-1}),
\]
and the latter belongs to $\Psm(X)$ by \cref{lem:larger-source}.
\end{proof}

For an ordinary point blowup, the nefness assertion also has an elementary proof: write $A_i=\pi^*L_i-a_iE$. A line in $E\simeq\Pj^{n-1}$ gives $a_i\geq0$, and every curve $C\subset Z$ has a strict transform $\widetilde C\not\subset E$, so $L_i\cdot C=A_i\cdot\widetilde C+a_iE\cdot\widetilde C\geq0$. Every mixed term containing $E$ is a one-cycle supported over a point and pushes to zero. Iteration therefore removes a terminal block of point blowups without losing nefness.

A curve blowup performed after a point blowup can change its exceptional divisor; the preceding argument removes such a block only when its entire exceptional image on the intermediate target is finite.

\begin{proposition}
Let $f:V\to W$ and $\rho:W\to X$ be projective birational morphisms of smooth projective $n$-folds, $n\geq3$. Put $Z=f(\operatorname{Exc}(f))$. Suppose $\rho(Z)$ is finite and
every curve contained in $Z$ has numerical class in the closed cone generated by curves of $W$ not contained in $Z$. Then $f_*A$ is ample for every ample rational divisor $A$ on $V$, and
\[
 (\rho f)_*(A_1\cdots A_{n-1})=\rho_*\bigl((f_*A_1)\cdots(f_*A_{n-1})\bigr).
\]
If $\rho$ contracts no divisor to a point, this class belongs to $\cA_{\mathrm{sm}}(X)$.
\end{proposition}
\begin{proof}
Put $A_0=f_*A$. The negativity lemma gives $F=f^*A_0-A\geq0$, with $F$ exceptional. Fix an ample divisor $H$ on $W$, and choose rational $\delta>0$ such that $A-\delta f^*H$ is ample. For a curve $C\not\subset Z$, its strict transform $\widetilde C$ is not contained in $F$, so
\[
 (A_0-\delta H)C=(A-\delta f^*H)\widetilde C+F\widetilde C\geq0.
\]
The numerical hypothesis extends this inequality to curves in $Z$. Thus $A_0-\delta H$ is nef, and $A_0$ is ample. Expand $A_i=f^*(f_*A_i)-F_i$. Every term containing some $F_i$ is represented by a one-cycle supported in $f^{-1}(Z)$, whose image in $X$ is finite; its pushforward is zero. The projection formula proves the identity. The last assertion is \cref{lem:larger-source}.
\end{proof}

For a blowup of a smooth curve $Z\subset W$, the numerical hypothesis holds if $Z$ deforms to another curve of the same class. Thus some point-fiber refinements can be removed regardless of the genus of their curve centres. The proposition requires the stated numerical hypothesis and does not eliminate arbitrary rigid centres.

\medskip
\noindent\emph{Small modifications.}
We also need the following statement for the minimal-model arguments of \cref{subsec:minimal-models}.

\begin{lemma}\label{lem:small-modifications}
Let $\phi:X\dashrightarrow X'$ be a small projective birational map, where $X$ is smooth and $X'$ is normal $\Q$-factorial. Let $B'$ be an ample $\Q$-Cartier divisor on $X'$. The numerical strict-transform class of $(B')^{n-1}$ on $X$ belongs to $\cA_{\mathrm{sm}}(X)$.
\end{lemma}

\begin{proof}
The map is an isomorphism $U\simeq U'$ with complements of codimension at least two. For $m$ sufficiently divisible, a general complete intersection of $n-1$ members of $|mB'|$ avoids $X'\setminus U'$. It is smooth and irreducible by Bertini on the smooth open set $U'$. Its image in $U\subset X$ is a smooth embedded curve with normal bundle $\cO_{C'}(mB')^{\oplus(n-1)}$, hence ample. Dividing its class by $m^{n-1}$ gives the desired class.
\end{proof}

\section{Cone equalities for special varieties and blowups}\label{sec:positive-cases}
We apply the preceding criteria to obtain cone equalities for special classes of varieties. We then compute movable cones for certain blowups and construct smooth curves with ample normal bundles in their exceptional divisors.

\subsection{Minimal-model criteria}\label{subsec:minimal-models}

Case~(i) below is a direct consequence of Fulger--Lehmann's construction \cite[Example~5.11]{FL17}; case~(ii) applies the minimal-model theory of \cite{BCHM10} to the same geometric criterion.

\begin{theorem}
Let $X$ be a smooth projective variety. If either
\begin{enumerate}
\item $X$ is a Mori dream space, or
\item there is an effective $\Q$-divisor $\Delta_0$ such that $(X,\Delta_0)$ is klt and $-(K_X+\Delta_0)$ is nef,
\end{enumerate}
then
\[
 \cA_{\mathrm{sm}}(X)=\Mov(X).
\]
The first case includes smooth projective toric varieties. The second includes the cases $-K_X$ nef and, in particular, $K_X\equiv0$.
\end{theorem}

\begin{proof}
\noindent\emph{Case (i): $X$ is a Mori dream space.}
Fulger--Lehmann \cite[Example~5.11]{FL17} show that the movable curve cone is generated, after taking closure, by positive products on projective small $\Q$-factorial modifications $X'$. Their proof realizes the relevant products as limits of ample complete intersections in the common isomorphism locus. By \cref{lem:small-modifications}, every such approximating class belongs to $\cA_{\mathrm{sm}}(X)$. This proves $\Mov(X)\subseteq\cA_{\mathrm{sm}}(X)$; the reverse inclusion follows from \cite[Theorem~2.4]{BDPP13}.

\medskip
\noindent\emph{Case (ii): $-(K_X+\Delta_0)$ is nef.}
It is enough to treat positive products $\langle D^{n-1}\rangle$ of big movable rational divisors $D$: these generate the movable curve cone after closure \cite[Example~5.11]{FL17}. Choose $D\sim_\Q A+E$, with $A$ ample and $E\geq0$. Write $N=-(K_X+\Delta_0)$ and choose $\varepsilon>0$ sufficiently small so that $(X,\Delta_0+\varepsilon E)$ is klt. The class $N+\varepsilon A$ is ample. A general small-coefficient effective
representative $G\sim_\Q N+\varepsilon A$ can be chosen so that
\[
\Theta=\Delta_0+\varepsilon E+G\quad\text{is klt},\qquad K_X+\Theta\sim_\Q\varepsilon D.
\]
The boundary $\Theta$ is big. By BCHM \cite{BCHM10}, an MMP with scaling for $(X,\Theta)$ terminates in a nef model for $D$.

This program has no divisorial contractions. Indeed, suppose that a step contracted a prime divisor $F$. Its contracted curves cover $F$ and belong to a single extremal ray $R$, on which the current transform of $D$ has negative intersection. For every effective movable rational divisor $M$, a suitable multiple has an effective representative not containing $F$. A general contracted curve in a family covering $F$ is not contained in that representative, so $M$ is nonnegative on $R$. By linearity and continuity, the same holds for every class in the closed movable divisor cone, contradicting the negativity of the current transform of $D$ on $R$. Small modifications preserve this closed movable cone, so the argument applies at every step. The resulting map $X\dashrightarrow X'$ is therefore small, and the transform $D'$ is nef and big.

Choose an ample divisor $H'$ on $X'$ and let $t\downarrow0$ through positive rationals. Then $D'+tH'$ is ample, so the strict-transform classes of $(D'+tH')^{n-1}$ belong to $\cA_{\mathrm{sm}}(X)$ by \cref{lem:small-modifications}. The birational invariance of the positive product on a small $D$-nef model gives
\[
\langle D^{n-1}\rangle=\lim_{t\downarrow0}\phi_*^{-1}[(D'+tH')^{n-1}].
\]
Taking closed cones proves the assertion.
\end{proof}

\subsection{Complete-intersection blowups and nef products}

We construct smooth curves inside the exceptional divisor of a blowup and use their numerical classes to determine its movable cone. Write $\CI(Y)$ for the closed cone generated by products of $n-1$
nef divisor classes on a smooth $n$-fold. General ample complete intersections and \cite[Theorem~2.4]{BDPP13} give
\begin{equation}
 \CI(Y)\subseteq\cA_{\mathrm{sm}}(Y)\subseteq\Mov(Y).
\end{equation}

\begin{theorem}\label{thm:ci-all-codim}
Let $X$ be a smooth projective $n$-fold with $\rho(X)=1$, and let $L$ be an ample Cartier divisor. Suppose that $Z$ is a smooth connected positive-dimensional scheme-theoretic complete intersection of $r$ ample Cartier divisors $D_1,\ldots,D_r$, where
$2\leq r\leq n-1$. Write
\[
 D_i\equiv a_iL,\qquad 0<a_1\leq\cdots\leq a_r.
\]
Let $\pi:Y=\operatorname{Bl}_ZX\to X$, put $H=\pi^*L$, and let $E$ be its exceptional divisor. In the curve coordinates $h=H\cdot\alpha$ and $e=E\cdot\alpha$, one has
\begin{align}
\cA_{\mathrm{sm}}(Y)=\Mov(Y)&=\{(h,e):h\geq0,\ 0\leq e\leq a_1h\},\label{eq:ci-all-wedge}\\
\Eff^{\,1}(Y)&=\operatorname{cone}(E,a_1H-E).\label{eq:ci-all-effective}
\end{align}
Every rational interior ray is represented by a smooth irreducible curve with ample normal bundle. Every movable class, including a boundary class, is a limit of positive rescalings of individual such curves. No smoothness assumption on the partial
intersections of the $D_i$ is required.
\end{theorem}

\begin{proof}
Put $m=\dim Z$ and $L_i=\cO_Z(D_i)$. The normal bundle splits as $N_{Z/X}=\bigoplus_{i=1}^rL_i$. With the quotient convention,
\[
E=\mathbb P_Z(N_{Z/X}^*),\qquad\cO_Y(E)|_E=\cO_E(-1).
\]
If $m=1$, set $B=Z$. If $m\geq2$, take a general smooth complete-intersection curve $B\subset Z$ of sufficiently ample divisors. Then $N_{B/Z}$ is a direct sum of ample line bundles. Write $L_i$ also for their restrictions to $B$, and set $h_B=\deg(L|_B)>0$.

Inside $E_B=E\times_ZB$ consider the ruled surface
\[
 R=\mathbb P_B(L_1^{-1}\oplus L_2^{-1})\subset E_B.
\]
This is the subbundle of quotients that kill the remaining summands. Let $p:R\to B$, $\xi=c_1(\cO_R(1))$, and $H_B=p^*c_1(L|_B)$. The divisor
\[
 Q=\xi+p^*c_1(L_2)
\]
is nef: it is the tautological class of $(L_2\otimes L_1^{-1})\oplus\cO_B$, whose summands have nonnegative degree. It is relatively ample over $B$. For rational $0<\varepsilon<a_1$, the divisor $A_\varepsilon=Q+\varepsilon H_B$ is ample. Indeed it has positive degree on every curve, and
\[
 A_\varepsilon^2=(a_2-a_1+2\varepsilon)h_B>0.
\]
Here we used $\xi^2=-(a_1+a_2)h_B$ and $\xi H_B=h_B$.

For sufficiently divisible $k$, choose a general smooth irreducible $C_\varepsilon\in|kA_\varepsilon|$. Its coordinates in $Y$ are
\begin{equation}\label{eq:ruled-slope}
 h=kh_B,\qquad e=-\xi C_\varepsilon=k(a_1-\varepsilon)h_B.
\end{equation}
The curve maps finitely onto $B$, with degree $k$. We verify the normal bundles along
\[
 C_\varepsilon\subset R\subset E_B\subset E\subset Y.
\]
First $N_{C_\varepsilon/R}$ has positive degree $k^2A_\varepsilon^2$. Next,
\[
 N_{R/E_B}=\bigoplus_{j=3}^r\bigl(\cO_R(1)\otimes p^*L_j\bigr).
\]
Each summand has degree $(a_j-a_1+\varepsilon)h>0$ on $C_\varepsilon$. The restriction of $N_{E_B/E}$ is the finite pullback of $N_{B/Z}$, hence ample. Finally $N_{E/Y}|_{C_\varepsilon}$ has degree $e>0$. Omit the zero-rank terms when $r=2$ or $m=1$.
Successive normal exact sequences, and stability of ampleness under extensions on a curve, prove that $N_{C_\varepsilon/Y}$ is ample.

Since $Z$ is smooth and connected, $N^1(Y)_\R=\R H\oplus\R E$. Both $E$ and the effective strict transform $\pi^*D_1-E\equiv a_1H-E$ pair nonnegatively with movable
classes. They constrain the movable cone to lie in the wedge in \eqref{eq:ci-all-wedge}. Equation~\eqref{eq:ruled-slope} supplies every rational interior ray of this wedge. Its closure therefore
equals both cones in \eqref{eq:ci-all-wedge}, and duality gives \eqref{eq:ci-all-effective}. Approximating any desired slope and rescaling the $h$-coordinate gives the assertion about individual curves, also on the two boundary rays. The zero class is obtained by rescaling any fixed curve to zero.
\end{proof}

\begin{remark}
The proof only uses a splitting $N_{Z/X}=\bigoplus L_i$ with $L_i\equiv a_iL|_Z$, and an effective divisor numerically $a_1L$ containing $Z$ to order one. Under these hypotheses the same conclusion holds without a global complete-intersection presentation of $Z$.
\end{remark}

\medskip
\noindent\emph{A codimension-two refinement.}
The next proposition strengthens the cone computation in \cref{thm:ci-all-codim} when the defining divisors meet transversely.

\begin{proposition}\label{prop:ci-two-cones}
In \cref{thm:ci-all-codim}, suppose $r=2$ and $D_1,D_2$ are smooth and transverse. Set $a=a_1$, $b=a_2$ and $N=\pi^*D_2-E$. Then
\[
 \Nef(Y)=\operatorname{cone}(H,bH-E),\qquad\CI(Y)=\cA_{\mathrm{sm}}(Y)=\Mov(Y).
\]
Every nonzero interior movable class is $P^{n-1}$ for an ample real divisor $P$. Every movable class is a limit of rescaled individual smooth ample-normal complete-intersection curves.
\end{proposition}

\begin{proof}
The strict transform $F=\pi^*D_1-E$ is isomorphic to $D_1$, because $Z$ is Cartier on $D_1$. Moreover $E|_F=D_2|_{D_1}$ and $N|_F$ is trivial. Numerically, $N\equiv F+(b-a)H$. Curves outside $F$ have nonnegative $N$-degree, and curves in $F$ have degree zero. Thus $N$ is nef. Testing $xH-yE$ on a ruling of $E$ and a curve in $F$ gives
$y\geq0$ and $x-by\geq0$, which proves the nef-cone formula.

Put $d=L^n$, $c=b-a$, $\gamma_0=H^{n-1}$ and $\gamma_1=H^{n-2}N$. The numerical identity $FN=0$ gives
\[
 (H\gamma_0,E\gamma_0)=(d,0),\qquad (H\gamma_1,E\gamma_1)=(bd,abd),\qquad N^2=cHN.
\]
The two nef products generate the entire movable wedge, proving the cone equality. Furthermore
\begin{equation}\label{eq:ci-power-parametrization}
(H+tN)^{n-1}=\gamma_0+f_c(t)\gamma_1,\quad f_c(t)=
 \begin{cases}
 ((1+ct)^{n-1}-1)/c,&c>0,\\
 (n-1)t,&c=0.
 \end{cases}
\end{equation}
As $t$ ranges over $(0,+\infty)$, so does $f_c(t)$. The divisor $H+tN$ is ample for every $t>0$: relative ampleness gives this for small positive $t$, and adding nef multiples of
$N$ gives all larger $t$. Equation~\eqref{eq:ci-power-parametrization}, together with positive scaling, therefore expresses every interior class as the $(n-1)$st power of an ample real divisor. Approximation by rational ample divisors and Bertini give the individual curves. Taking limits also gives both boundary rays.
\end{proof}

\begin{example}[A nef boundary without a semiample representative]

Let $X$ be an abelian threefold with $\rho(X)=1$, let $L$ be sufficiently ample, and take a nontorsion $P\in\operatorname{Pic}^0(X)$. Choose smooth transverse divisors $D_1\in|L|$ and $D_2\in|L\otimes P|$, with smooth connected intersection $Z$. On $Y=\operatorname{Bl}_ZX$, the nef ray $H-E$ has no semiample line-bundle representative, although \cref{prop:ci-two-cones} applies.

Indeed, let $F_i$ be the disjoint strict transforms and put $N=\pi^*(L\otimes P)\otimes\cO_Y(-E)$. Then
\[
 N|_{F_1}\simeq\cO_{F_1},\qquad N|_{F_2}\simeq P|_{D_2}.
\]
Every line bundle on the positive numerical ray of $N$ has the form $N^{\otimes m}\otimes\pi^*Q$ with $m\in\Z_{>0}$ and $Q\in\operatorname{Pic}^0(X)$: the exceptional coefficient forces $m$ to be integral, and the N\'eron--Severi group of an abelian
variety is torsion-free. A semiample numerically trivial line bundle is torsion. Restriction to $F_1$, and the Lefschetz isomorphism $\operatorname{Pic}^0(X)\simeq\operatorname{Pic}^0(D_1)$, would therefore force $Q$ to be torsion. Restriction to $F_2$ would then be nontorsion and semiample, a contradiction. Thus the proof of the cone equality does not require semiampleness on the nef boundary.
\end{example}

\subsection{Curve centres beyond complete intersections}

\begin{theorem}\label{thm:non-ci-centre}
Let $X$ be a smooth projective threefold with $\rho(X)=1$, let $L$ be an ample Cartier divisor, and let $D\equiv aL$ be an effective Cartier divisor. Suppose $Z\subset D$ is a smooth connected curve and $D$ is smooth along $Z$. If
\begin{equation}\label{eq:normal-degree-criterion}
 l:=\deg N_{Z/D}\ \geq\ m:=D\cdot Z>0,
\end{equation}
then, for $Y=\operatorname{Bl}_ZX$ and the coordinates $h=(\pi^*L)\cdot\alpha$, $e=E\cdot\alpha$,
\[
\cA_{\mathrm{sm}}(Y)=\Mov(Y)=\{(h,e):h\geq0,\ 0\leq e\leq ah\}.
\]
Every rational interior ray has an individual smooth ample-normal representative, and every movable class is a limit of rescaled individual such curves. If $D$ is smooth, the first degree in \eqref{eq:normal-degree-criterion} is $Z_D^2$. In particular, if $D$ is a smooth ample surface and $A$ is an ample divisor on $D$, then for all sufficiently large integers $s$ the blowup along a general smooth $Z\in|sA|$ satisfies the same cone equality and approximation by individual curves.
\end{theorem}

\begin{proof}
Write $L_Z=N_{Z/D}$ and $M_Z=\cO_Z(D)$. The normal sequence gives
\[
0\longrightarrow L_Z\longrightarrow N_{Z/X}\longrightarrow M_Z\longrightarrow0.
\]
On $E=\mathbb P_Z(N_{Z/X}^*)$, with projection $p$ and $\xi=c_1(\cO_E(1))$, put $Q=\xi+p^*c_1(L_Z)$.
The sequence
\[
0\longrightarrow L_Z\otimes M_Z^{-1}\longrightarrow N_{Z/X}^*\otimes L_Z\longrightarrow\cO_Z\longrightarrow0
\]
has nef end terms by $l\geq m$. An extension of nef bundles on a smooth curve is nef, so $Q$ is nef. Put $H_E=p^*c_1(L|_Z)$ and $d_Z=L\cdot Z$, so $m=ad_Z$. For rational $0<\varepsilon<a$, $Q+\varepsilon H_E$ is ample: it has positive degree on every curve and square $l-m+2\varepsilon d_Z>0$.

A general member $C_\varepsilon$ of a sufficiently divisible $|k(Q+\varepsilon H_E)|$ is smooth irreducible and satisfies
\[
H\cdot C_\varepsilon=kd_Z,\quad E\cdot C_\varepsilon=k(m-\varepsilon d_Z)\quad\textrm{and}\quad C_\varepsilon^2|_E=k^2(l-m+2\varepsilon d_Z)>0.
\]
These identities use $\xi^2=-(l+m)$ and $\xi p^*c_1(L_Z)=l$. Both $N_{C_\varepsilon/E}$ and $N_{E/Y}|_{C_\varepsilon}$ have positive degree. Their extension $N_{C_\varepsilon/Y}$ is ample. The slopes are $e/h=a-\varepsilon$. Finally $E$ and the effective strict transform $\pi^*D-E\equiv aH-E$ bound the movable cone by the asserted wedge. The curves just constructed fill its rational interior rays, and closure and rescaling finish the proof.

For the final assertion, let $Z\in|sA|$ be general on a smooth ample surface $D$. Bertini and ampleness make $Z$ smooth and connected, and the numerical hypothesis follows for $s\gg0$ from
\[
  \deg N_{Z/D}=s^2A^2\geq s(D|_D\cdot A)=D\cdot Z>0.
\]
\end{proof}

\section{Moving families and finite descent}\label{sec:moving-families}
Throughout this section, $X$ is smooth projective of dimension $n\ge3$. We apply \cref{lem:larger-source} to moving families, then prove a descent statement for finite morphisms.

\subsection{Moving families and fixed-error approximation}
A family over a curve yields an approximation whose error is independent of the integer $m$.

\begin{proposition}\label{prop:fixed-error}
Let $p:V^{n+1}\to B$ be a projective morphism to a smooth projective curve, with $V$ smooth. Let $F:V\to X^n$ be dominant and have no prime divisor mapped to a point. Choose very ample $H_1,\ldots,H_{n-1},H_0$ on $V$, and write
\[
 \beta=F_*(H_1\cdots H_{n-1}[V_b])
\]
for the general fiber class. Then there are smooth curves $C_m\subset X$ with ample normal bundle and
\begin{equation}
 [C_m]=md\,\beta+\eta,
\end{equation}
where $d>0$ and $\eta\in N_1(X)$ do not depend on $m$.
\end{proposition}

\begin{proof}
Choose a very ample line bundle $L$ on $B$ of degree $d$. The divisor $H_0+mp^*L$ is very ample for every $m\geq0$, since it is a very ample line bundle tensored with a globally generated one. Take a general complete intersection of $H_1,\ldots,H_{n-1}$ and
$H_0+mp^*L$. By \cref{lem:larger-source} its image $C_m$ is smooth with ample normal bundle. The projection formula gives
\[
[C_m]=F_*(H_1\cdots H_{n-1}H_0)+mF_*(H_1\cdots H_{n-1}p^*L)=\eta+md\,\beta.
\]
\end{proof}

\begin{theorem}\label{thm:moving-geometries}
Let $X$ be a smooth projective $n$-fold, $n\geq3$. Then $\cA_{\mathrm{sm}}(X)=\Mov(X)$ in either of the following cases:
\begin{enumerate}
\item $X$ admits an action of a positive-dimensional abelian variety $G$ such that every $G$-orbit has positive dimension. In particular, this applies to $X=Z\times G$ for any smooth projective $Z$ with $\dim(Z\times G)\geq3$.
\item $X=Z\times\mathbb P^r$, where $Z$ is smooth projective, $\dim Z\geq1$, and $r\geq1$.
\end{enumerate}
\end{theorem}

\begin{proof}
\noindent\emph{Case (i): a moving abelian action.}
Fix a generator $\beta=\mu_*(A_1\cdots A_{n-1})$ and let $P$ be the finite set of point images of exceptional prime divisors of $\mu:Y\to X$. Choose a smooth projective curve $B$ and a morphism $a:B\to G$ such that $b\mapsto a(b)\cdot p$ is nonconstant for every
$p\in P$. A general complete-intersection curve in $G$ has this property; if $G$ is a curve, take $B=G$.

Put $V=Y\times B$ and
\[
 F(y,b)=a(b)\cdot\mu(y).
\]
If a prime divisor of $V$ dominated $B$ and mapped to a point of $X$, its general fiber in $Y$ would be a divisor mapped by $\mu$ to one of the finitely many points of $P$. This would force one of the chosen orbit maps to be constant. A vertical divisor is a full fiber $Y\times\{b\}$ and dominates $X$. Hence no divisor maps to a point. Choose very ample divisors on $V$ restricting to positive multiples of the $A_i$ on its fibers. Since the connected action preserves numerical curve classes, \cref{prop:fixed-error} gives smooth curves with ample normal bundles whose rescaled classes tend to $\beta$. Taking closed cones proves the equality.

\medskip
\noindent\emph{Case (ii): a projective-space factor.}
Fix a strongly movable generator $\beta=\mu_*(A_1\cdots A_{n-1})$, where $\mu:Y\to X$ is birational, $Y$ is smooth, and the $A_i$ are very ample. Let $P$ be the finite set of point-exceptional images of $\mu$, and let $h:Y\to\mathbb P^r$ be the second projection composed with $\mu$. Choose a hyperplane $H\subset\mathbb P^r$ avoiding the projection of $P$, and a point $p\notin H$ avoiding the projection of
$P$, such that $D=h^{-1}(H)$ and $K=h^{-1}(p)$ are smooth. Choose coordinates in which $p=[1:0:\cdots:0]$ and $H=\{x_0=0\}$.

The one-parameter action
\[
\lambda_t[x_0:x_1:\cdots:x_r]=[t x_0:x_1:\cdots:x_r]
\]
moves every point of $P$. Over $\mathbb G_m$ it defines $F_t(y)=(\operatorname{pr}_Z\mu(y),\lambda_t(h(y)))$. The compactified rational map on $Y\times\mathbb P^1$ is
\[
(y,[u:v])\dashrightarrow(\operatorname{pr}_Z\mu(y),[u h_0(y):v h_1(y):\cdots:v h_r(y)]).
\]
Its base loci are the disjoint smooth centres $K\times\{0\}$ and $D\times\{\infty\}$. Blowing them up gives a smooth projective $V$ with a morphism $F:V\to X$. In local coordinates the base ideals are $(u,w_1,\ldots,w_r)$ and $(v,w_0)$, respectively. The corresponding exceptional divisors map onto projective spaces or lines in the $\mathbb P^r$ factor. The strict transforms of the boundary fibers map onto $Z\times H$ and $Z\times\{p\}$; the latter has positive dimension by assumption. A horizontal divisor
mapped to a point would force a point of $P$ to have a constant $\lambda_t$-orbit, contrary to the choice of $H,p$. Thus no prime divisor of $V$ maps to a point.

The map $V\to\mathbb P^1$ has general fiber $Y$. Let $q:V\to Y\times\mathbb P^1$ be the blowdown, with exceptional divisors $E_0$ and $E_\infty$. The line bundle $\cO_V(-E_0-E_\infty)$ is $q$-ample. For each $i$, the divisor
\[
 T_i=\operatorname{pr}_Y^*A_i+
 \operatorname{pr}_{\mathbb P^1}^*\cO_{\mathbb P^1}(1)
\]
is ample on $Y\times\mathbb P^1$. Thus $M_i=q^*(k_iT_i)-E_0-E_\infty$ is ample for sufficiently large integers $k_i$. Choose positive integers $s_i$ so that $H_i=s_iM_i$ is very ample. On the fiber at $t=1$, it restricts to $s_ik_iA_i$. Choose any very ample $H_0$ on $V$ and put $c=\prod_i s_ik_i$. By \cref{prop:fixed-error}, there are smooth ample-normal curves with $[C_m]=mdc\,\beta+\eta$, where $d>0$ and $\eta$ are independent of $m$. Thus $[C_m]/(mdc)\to\beta$.
\end{proof}

\subsection{Normal-source cones and finite descent}
We now isolate the cone obtained from smooth sources with no divisor mapped to a point. This construction is preserved by composition with a finite morphism.

\begin{definition}
Let $\cG_1(X)$ be the closed convex cone generated by classes $F_*(H_1\cdots H_{N-1})$, where $F:V^N\to X^n$ ranges over dominant morphisms of smooth projective varieties, $N\geq n\geq3$, no prime divisor of $V$ is mapped to a point, and all $H_i$ are very ample.
\end{definition}

By \cref{lem:larger-source},
\begin{equation}\label{eq:good-inclusions}
 \cG_1(X)\subseteq\cA_{\mathrm{sm}}(X)\subseteq\Mov(X).
\end{equation}
The proof of \cref{thm:moving-geometries} shows that $\cG_1(X)=\Mov(X)$ in those cases.

\begin{proposition}
Let $f:Y\to X$ be finite and surjective between smooth projective varieties of dimension $n\geq3$. If $\cG_1(Y)=\Mov(Y)$, then $\cG_1(X)=\Mov(X)$, and hence $\cA_{\mathrm{sm}}(X)=\Mov(X)$.
\end{proposition}

\begin{proof}
The map $f$ is flat by miracle flatness. For $\alpha\in\Mov(X)$, its flat pullback is movable: if $D$ is an effective divisor on $Y$, then $D\cdot f^*\alpha=f_*D\cdot\alpha\geq0$, and BDPP duality applies. By hypothesis, $f^*\alpha\in\cG_1(Y)$. A finite morphism cannot map a positive-dimensional image to a point, so composition with $f$ sends the defining generators of $\cG_1(Y)$ to those of
$\cG_1(X)$. Therefore
\[
 (\deg f)\alpha=f_*f^*\alpha\in\cG_1(X).
\]
Use \eqref{eq:good-inclusions} for the reverse inclusion.
\end{proof}

\section{Conditional methods for smooth replacement}\label{sec:conditional-replacement}
We give two conditional methods for constructing smooth representatives when a birational model has divisors contracted to points or singular blowup centres. The first assumes finite generation of a saturated section algebra; the second assumes a global smoothing of the centre.

\subsection{Replacement via point-saturated models}\label{subsec:point-saturated}

Finite generation of the following saturated section algebra allows us to remove divisors contracted to points while preserving the pushed-forward complete-intersection class.

Let $\mu:Y\to X$ be a birational morphism of smooth projective varieties, and let $A$ be an ample line bundle on $Y$. Write $P\subset X$ for the finite set of images of prime divisors contracted to points, and let $j:U=X\setminus P\hookrightarrow X$. Define the point-saturated graded algebra
\begin{equation}\label{eq:saturated-algebra}
\cS(\mu,A;P)=\bigoplus_{m\geq0}j_*\left(\left.\mu_*\cO_Y(mA)\right|_U\right).
\end{equation}

\begin{proposition}\label{prop:saturated-replacement}
If the graded $\cO_X$-algebra \eqref{eq:saturated-algebra} is finitely generated, then, after passing to a Veronese subalgebra, there are a normal projective variety $W$, a projective birational morphism $\rho:W\to X$, and an ample $\Q$-Cartier divisor $B$ on $W$ with
\begin{align}
 W|_U&\simeq Y|_U,\\
 \text{no prime divisor of }W&\text{ is contracted to a point},\\
 \rho_*(B^{n-1})&=\mu_*(A^{n-1}).\label{eq:saturated-equality}
\end{align}
Consequently $\mu_*(A^{n-1})\in\cA_{\mathrm{sm}}(X)$.
\end{proposition}

\begin{proof}
Choose $k>0$ such that the $k$-th Veronese is generated in degree one. Put $W=\Proj_X\cS^{(k)}$ and let $L=\cO_W(1)$. Set $B=k^{-1}L$ to preserve the normalization of the original class $A^{n-1}$. The algebra over $U$ is the relative section algebra of the $\mu$-ample line bundle $A$, so
$W|_U\simeq Y|_U$ and $B$ is relatively ample. The section algebra is integrally closed as an intersection of the divisorial valuation conditions defining $\mu_*\cO_Y(mA)$ on $U$, so $W$ is normal.

Suppose that a prime divisor $E\subset W$ mapped to $p\in P$. Since $W$ is normal, the divisorial sheaf $\cO_W(E)$ is coherent, even though $E$ need not be Cartier. Put
\[
 \mathcal Q_E=\cO_W(E)/\cO_W.
\]
This is a nonzero coherent sheaf supported on $E$. For sufficiently large $m$ divisible by $k$, the divisor $mB$ is Cartier, relative Serre vanishing gives $R^1\rho_*\cO_W(mB)=0$, and $\mathcal Q_E\otimes\cO_W(mB)$ is relatively generated. In particular, its pushforward is nonzero. Twisting the defining exact sequence of $\mathcal Q_E$ and pushing forward therefore gives
\[
 \frac{\rho_*\cO_W(mB+E)}{\rho_*\cO_W(mB)}
 \simeq\rho_*\bigl(\mathcal Q_E\otimes\cO_W(mB)\bigr)\ne0.
\]
Both pushforwards are torsion-free subsheaves of the same sheaf of rational sections and agree over $U$. For $m\gg0$ divisible by $k$, the construction by relative Proj and the saturated definition give
\[
 \rho_*\cO_W(mB)=j_*\bigl(\rho_*\cO_W(mB)|_U\bigr).
\]
Consequently $\rho_*\cO_W(mB+E)$ is already contained in $\rho_*\cO_W(mB)$, contradicting the nonzero quotient. No prime divisor over $U$ is contracted to a point, since $W|_U\simeq Y|_U$ and $P$ contains all such images for $\mu$.

On a common resolution $p:V\to Y$, $q:V\to W$, inclusion of the original section algebra in its saturation yields
\[
 q^*B=p^*A+F,
\]
where $F\geq0$ is supported over $P$. Take an ample divisor $H$ on $X$ and $\varepsilon>0$ small enough that $A-\varepsilon\mu^*H$ is ample. Then
\[
 q^*(B-\varepsilon\rho^*H)=p^*(A-\varepsilon\mu^*H)+F.
\]
On curves not mapped into $P$, the right-hand side intersects the strict transform nonnegatively; on curves over $P$, relative ampleness of $B$ gives positivity. Hence
$B-\varepsilon\rho^*H$ is nef. Since $B+k\rho^*H$ is ample for $k\gg0$, expressing $B$ as a convex combination of these two classes proves that $B$ is ample.

Finally, expand $(q^*B)^{n-1}-(p^*A)^{n-1}$ using $F$. All its terms are one-cycles supported over $P$, hence have zero pushforward to $X$. This yields \eqref{eq:saturated-equality}. If $n\geq3$, a general complete-intersection curve of high multiples of $B$ on the normal model $W$ avoids the codimension-two singular locus and the inverse image of $P$. Apply \cref{lem:larger-source} to obtain a smooth embedded image with ample normal bundle.

If $n=2$, a positive-dimensional fiber of $\rho$ would contain a curve, hence a prime divisor contracted to a point. Thus $\rho$ is proper and quasi-finite, hence finite; being birational onto the normal surface $X$, it is an isomorphism. Under $W\simeq X$, a general member $C\in|mB|$ for sufficiently divisible $m\gg0$ is smooth and irreducible, with ample normal bundle $N_{C/X}\simeq\cO_C(mB)$. Since $[C]/m=\rho_*B=\mu_*A$, the conclusion follows also in dimension two.
\end{proof}

\begin{remark}[The finite-generation hypothesis]
Locally at $p\in P$, in a trivialization in which the graded algebra on $U$ is generated by an ideal $I$, the graded pieces are $(I^m:\mathfrak m_p^\infty)$. They need not form a finitely generated graded algebra. The analytic-spread criterion and non-Noetherian symbolic Rees examples in \cite{CS21} give obstructions to replacing a fixed punctured polarized model in this way. Non-finite generation does not imply that the numerical
curve class has no smooth ample-normal representative: a different birational model may represent the same ray.
\end{remark}

\subsection{Replacement via deformation of centres}\label{subsec:deformations}

The second method uses a global deformation of a singular centre to a smooth one, without assuming finite generation of the point-saturated section algebra.

\begin{proposition}\label{prop:smoothable-center}
Let $X$ be a smooth projective $n$-fold, $n\geq3$, and let $Z_0\subset X$ be a reduced pure-dimensional subscheme of dimension between $1$ and $n-2$, with finite singular locus. Suppose $Z_0$ occurs as the special fiber of a flat family $\mathcal Z\subset X\times T$ over a smooth connected pointed curve, whose general fiber $Z_t$ is smooth. If
\[
 \pi_0:W_0=\operatorname{Bl}_{Z_0}X\to X,\qquad B_0=\pi_0^*D-bE_0
\]
is ample (where $b>0$ is rational), then for general $t$ the divisor $B_t=\pi_t^*D-bE_t$ on $W_t=\operatorname{Bl}_{Z_t}X$ is ample and
\begin{equation}\label{eq:deformation-intersection}
 (\pi_t)_*(B_t^{n-1})=(\pi_0)_*(B_0^{n-1}).
\end{equation}
In particular, this class belongs to $\cA_{\mathrm{sm}}(X)$, and a positive multiple is represented by a single smooth ample-normal curve. When $n=3$, the common curve class is
\[
  (\pi_0)_*(B_0^2)=D^2-b^2[Z_0]\in\Asm(X).
\]
\end{proposition}

\begin{proof}
Form the total blowup
\[
q:\mathscr W=\operatorname{Bl}_{\mathcal Z}(X\times T)\longrightarrow X\times T,\qquad p=\operatorname{pr}_T\circ q,
\]
and put $\mathcal B=q^*\operatorname{pr}_X^*D-b\mathcal E$. The Rees algebra is a subalgebra of the domain $\cO_{X\times T}[t]$, so its local rings are torsion-free over the DVR of the base curve. Therefore $p$ is flat. Its special fiber contains $W_0$ with multiplicity one. Away from $P=\operatorname{Sing}(Z_0)$, the family of centres is smooth along the special fiber, with regular immersion into $X\times T$; formation of the blowup commutes with this base change. Every other component of the special fiber is thus supported over a point of $P$.

The restriction of $\mathcal B$ to $W_0$ is $B_0$, which is ample. On any additional special-fiber component, the pullback of $D$ is trivial because the component maps to a point of $X$, while $\cO_{\mathscr W}(-\mathcal E)$ is ample on that component as a fiber of the blowup over $X\times T$. Hence $\mathcal B$ is ample on every reduced component, and therefore on the full special fiber. Openness of fiberwise ampleness gives ampleness of $B_t$ for general $t$.

For a Cartier divisor $H$ on $X$, flatness and the fiberwise Hilbert polynomial make
\[
 \deg\bigl(q^*\operatorname{pr}_X^*H\cdot\mathcal B^{n-1}\cdot[\mathscr W_t]\bigr)
\]
independent of $t$. The factor pulled back from $X$ vanishes on every extra component supported over a point. Thus the special-fiber intersection reduces to the intersection on $W_0$, proving \eqref{eq:deformation-intersection}. The blowup along a smooth positive-dimensional centre contracts no prime divisor to a point. Applying \cref{lem:larger-source} to $W_t\to X$ gives the asserted smooth curve.

When $n=3$, the standard blowup intersection identities away from the finite singular set of $Z_0$ give $(\pi_0)_*E_0=0$ and $(\pi_0)_*(E_0^2)=-[Z_0]$. Any remaining one-cycle supported over those points pushes forward to zero. Expanding $B_0^2$ yields
$(\pi_0)_*(B_0^2)=D^2-b^2[Z_0]$, as asserted.
\end{proof}

\begin{remark}[Global smoothing is essential]
A local embedded smoothing need not extend to a global family on a prescribed projective $X$. Nor does smoothing a thickened ideal while preserving its Hilbert polynomial guarantee preservation of the exceptional-square contribution. For example, in $\C[[x,y]]$, the ideal $(x,y)^r$ has colength $r(r+1)/2$ but Hilbert--Samuel multiplicity $r^2$. A degeneration argument must therefore distinguish colength from Hilbert--Samuel multiplicity.
\end{remark}

\section*{Acknowledgments}

The construction of the examples benefited from assistance by the GPT-6.1-Sol model. The authors subsequently checked the argument for mathematical correctness.

\end{document}